%% file: ECSAA_noTop.tex
\documentclass[11pt, oneside, usenames, dvipsnames, svgnames, table, final]{amsart}

\usepackage[T1]{fontenc}
\usepackage[utf8]{inputenc}

\usepackage[foot]{amsaddr}

\newcommand{\iEnd}{\underline{\mathrm{End}}}
\newcommand{\transf}{\operatorname{transf}}
\newcommand{\orth}{\operatorname{orth}}
\newcommand{\symp}{\operatorname{symp}}

\usepackage{fourier}

\input{header}

\author{Uriya First${}^{\lowercase{a}}$ \and Ben Williams${}^{\lowercase{b}}$}
\address{${}^a$University of Haifa, 199 Abba Khoushy Avenue, Haifa, Israel\\ ${}^b$The University of British Columbia, 1984 Mathematics Rd, Vancouver BC V6T 1Z2, Canada}
\email{ufirst@univ.haifa.ac.il}
\email{ tbjw@math.ubc.ca}
\usepackage[backref=true,style=alphabetic,citestyle=alphabetic,url=false,backend=biber]{biblatex}
\newcommand{\benw}[2][]{\ifdraft{\todo[linecolor=Green,backgroundcolor=Green!25,bordercolor=Green,#1]{#2---Ben W.}}{}}
\newcommand{\uriyaf}[2][]{\ifdraft{\todo[linecolor=Red,backgroundcolor=Red!25,bordercolor=Red,#1]{#2---Uriya F.}}{}}

\begin{document}   

\title{Counterexamples in Involutions of Azumaya Algebras}

\thanks{
The first author is supported by an ISF grant no.\ 721/24. 
The second author acknowledges the support of the Natural Sciences and Engineering Research Council of Canada (NSERC), RGPIN-2021-02603.}

\begin{abstract}
  Suppose $A$ is an Azumaya algebra over a ring $R$ and $\sigma$ is an involution of $A$
  extending an order-$2$ automorphism $\lambda:R\to R$. We say $\sigma$ is
  \emph{extraordinary} if there does not exist a  Brauer-trivial Azumaya algebra $\End_R(P)$ over $R$ carrying an involution $\tau$
  so that $(A, \sigma)$ and $(\End_R(P), \tau)$ become isomorphic over some
  faithfully flat extension of the fixed ring of $\lambda:R\to R$. 
  We give, for the first time, an example of such an algebra
  and involution. 
  We do this by finding suitable cohomological obstructions and showing they do not always vanish.

  We also give an example  of a commutative ring  $R$ with involution  $\lambda$ so that the scheme-theoretic  fixed locus $Z$ of
  $\lambda:\Spec R\to\Spec R$ is disconnected, but such that every Azumaya algebra  over $R$ with involution
  extending $\lambda$ is  either orthogonal at every point  of $Z$,
  or symplectic at every point of $Z$.
  No examples of this kind were previously known.
%  Suppose $A$ is an Azumaya algebra over a ring $R$ and $\sigma$ is an involution of $A$. We say $\sigma$ is
%  ``extraordinary'' if there does not exist a Brauer-trivial algebra $\End_R(P)$ over $R$ carrying an involution $\tau$
%  so that $(A, \sigma)$ is locally isomorphic to $(\End_R(P), \tau)$. This paper exhibits an example of such an algebra
%  and an involution, answering a question of \cite{First2020}. The method is by finding cohomological obstructions and showing they do not always vanish.comparison with the analogous problem in
%  algebraic topology.
%
%  We also give examples of commutative rings $R$ with involutions $\lambda$ so that the scheme-theoretically fixed loci
%  $\Spec(R)^\lambda$ are disconnected, but that every Azumaya algebra with involution over $R$ is of either of constant
%  orthogonal or symplectic type at all points of $\Spec(R)^\lambda$. This answers another question of
%  \cite{First2020}.
\end{abstract}
\maketitle

\section{Introduction}
\label{sec:intro}

Suppose $R$ is a commutative ring in which $2$ is a unit, and $A$ is an Azumaya algebra of degree $n$ over $R$,
i.e., a separable $R$-algebra with centre $Z(A) = R\cdot 1_A$. Suppose $\sigma : A \to A$ is an involution: an order-$2$ additive
self-map that reverses the order of multiplication. Restricting $\sigma$ to the centre, we obtain an order-$2$
automorphism $\lambda :R \to R$. We also say that $\sigma$ is a \emph{$\lambda$-involution}. 

An isomorphism of algebras with involution $\phi: (A, \sigma) \to (B, \tau)$ is an isomorphism of algebras that is also
compatible with the involution: $\phi(\sigma(a)) = \tau(\phi(a))$ for all $a \in A$.

Write $R^\lambda$ for the  fixed subring of $\lambda:R\to R$. 
In this introduction, we work over the fpqc site of $\Spec (R^{\lambda})$ (as defined in \cite[\href{https://stacks.math.columbia.edu/tag/03NV}{Tag 03NV}]{SP23}, see also Section \ref{sec:flat-topologies} below).
By saying that two objects defined over $R^{\lambda}$
are \emph{locally isomorphic}, we mean that they become isomorphic after base-change along some faithfully flat ring extension $R^\lambda\to S$: when there is a risk of confusion, we may clarify by saying \emph{fpqc-locally isomorphic over $R^\lambda$}.
For example,   two $R$-algebras with $\lambda$-involution
$(A, \sigma)$, $(B, \tau)$ are  locally isomorphic  if  there is some faithfully flat homomorphism of rings $R^\lambda
\to S$ for which the extensions $(A_S, \sigma_S)$ and $(B_S, \sigma_S)$ are isomorphic as $R_S$-algebras with involution. Here,
$A_S$ denotes  $S \tensor_{R^\lambda} A$ and $\sigma_S$ denotes $\id_S\otimes \sigma$.
%the obvious extension of $\sigma$ to an
%involution of this $S$-algebra.

\begin{remark}
  Later in the paper, we chiefly work over the \'etale site of $\Spec(R^\lambda)$, because this is
  the site studied in the bulk of \cite{First2020}---a source on which we rely extensively. When
  $2$ is invertible, \'etale- and fpqc-local isomorphism for Azumaya algebras with
  $\lambda$-involution coincide by virtue of our Proposition \ref{pr:etaleLocalIsoIsFPQC}.%
%  In Section \ref{sec:gen}, we introduce the notion of \'etale-locality, where we consider only the restricted case of
%  \'etale extensions $R^\lambda \to S$, which was the concept studied in the bulk of \cite{First2020}. When $2$ is
%  invertible, \'etale- and fpqc-local isomorphism for Azumaya algebras with $\lambda$-involution coincide by virtue of
%  Proposition \ref{pr:etaleLocalIsoIsFPQC}. 
\end{remark}

Upon forgetting the involution, an Azumaya algebra $A$ of degree $n$ over $R$ is fpqc-locally isomorphic over $R$ to $\Mat_{n \times n}(R)$.  
%That is, up to fpqc-local isomorphism, there is only one Azumaya $R$-algebra of each degree. 
This is also true
fpqc-locally 
%such an isomorphism also exists fpqc-locally  
over   $R^\lambda$ \cite[Thms.~4.28, 4.35]{First2020}. 
One might therefore hope that every Azumaya algebra with $\lambda$-involution $(A, \sigma)$ should be  locally isomorphic over $R^\lambda$ to $(\Mat_{n \times n}(R), \tau)$ for some choice of $\lambda$-involution
$\tau$. Strong evidence that this is not the case is given by \cite[Example 7.7]{First2020}, where it is shown that some $(A, \sigma)$ exists that is not locally isomorphic to $(\Mat_{n \times n}(R), \tau)$ if $\tau$ is constrained to be given by a composite of transposition, application of $\lambda$ and conjugation by some $x \in \GL_n(R)$---that is, there are examples where the involution $\tau$, if it existed, would have to be esoteric. Involutions
$\sigma$ for which  $\tau:\Mat_{n \times n}(R)\to \Mat_{n \times n}(R)$ exists and has the said special decomposition 
are called \emph{ordinary};
this and related  definitions to follow are special cases of
\cite[Def.~6.16]{First2020}. 
%In the well-studied situations where $\lambda=\id_R$
%or $R$ is quadratic \'etale over $R^\lambda$, all involutions are ordinary.

One might then hope that all Azumaya algebras with $\lambda$-involution $(A, \sigma)$ are locally isomorphic %\benw{added ``in the fpqc topology on $\Spec R^\lambda$ here.} in the fpqc topology on $\Spec(R^\lambda)$ 
over $R^\lambda$ 
to some Brauer-trivial Azumaya algebra with a $\lambda$-involution---a substantially weaker statement than that in the paragraph above. Specifically, we say $(A, \sigma)$, or just $\sigma$, is \emph{semiordinary} if there exists a finitely generated faithful projective $R$-module $P$ and an involution $\tau: \End_R(P) \to \End_R(P)$ so that $(\End_R(P), \tau)$ is locally isomorphic to $(A, \sigma)$.  

We call the involution $\sigma$ \emph{extraordinary} if it is not semiordinary. 
%(these definitions are special cases of
%\cite[Def.~6.16]{First2020}). 
In our first main result, Theorem \ref{th:AisExtraordinary}, we show that there exists an Azumaya algebra over a
ring $R$ with an extraordinary involution, dashing the hopes we articulated above, and
answering \cite[Problem~6.25]{First2020}. Philosophically, existence of an
extraordinary involution shows that the classification of Azumaya algebras with involution up to local isomorphism
cannot be reduced to questions about projective modules and hermitian forms on them, much less to classification of
involutions on matrix algebras.

\benw{added material}\uriyaf{Edited this remark.}
\begin{remark}
The discussion so far concerns Azumaya algebras \textit{per se} rather than their Brauer classes, but in fact if $(A,\sigma), (A', \sigma')$ are two Azumaya algebras with $\lambda$-involution over $R$,
%for some fixed $\lambda: R \to R$, 
and if $A$ and $A'$ are Brauer-equivalent, then $ \sigma $ is extraordinary (resp.\ semiordinary) if and only if $ \sigma' $ is:
see Proposition \ref{pr:supportExtraordinaryIsBrauerInv}. 
On the other hand, it may happen that $\sigma$ is ordinary and $\sigma'$ is not;
see Remark~\ref{rm:ordinaryBreq}(ii).
%When $\sigma'$ is moreover the adjoint involution of a $1$-hermitian space over $(A,\sigma)$,
%the involution $\sigma$ is ordinary if and only if $\sigma'$ is. All of this is shown in Section~\ref{sec:ordinary}.
\end{remark}
\benw{end of added material}

\uriyaf{being of change. the original paragraph is commented out.}  In order to present our next
main result, we recall the concept of coarse type of  $(A,\sigma)$ from \cite[\S 5.2]{First2020}.
Let $Z$ denote the scheme-theoretically fixed locus of $\lambda^*:\Spec R\to \Spec R$; see
Section~\ref{sec:CT} for its definition. Then $Z$ is a closed subscheme of $\Spec R$, and
$\lambda^*$ restricts to the identity on $Z$.  Moreover, writing $A_Z$ for the restriction of $A$ to
$Z$, the involution $\sigma$ restricts to an involution of the first kind $\sigma_Z:A_Z\to A_Z$. The
scheme $Z$ is therefore the disjoint union of two closed subsets, $Z=Z_{\orth}\cup Z_{\symp}$, such
that $\sigma$ is orthogonal when restricted to $Z_{\orth}$ and symplectic when restricted to
$Z_{\symp}$. This information can be neatly packed into a continuous section $c_\sigma:Z\to\{\pm1\}$
taking the value $1$ on $Z_{\orth}$ and $-1$ on $Z_{\symp}$; we call $c_\sigma$ the \emph{coarse
  type} of $\sigma$.

Two Azumaya algebras with $\lambda$-involution of the same degree over $R$  have the same coarse type if and only if they are locally isomorphic (\cite[Prop.~5.32]{First2020}).
Moreover, the coarse type of an involution determines whether the involution is ordinary, semiordinary, or
extraordinary (\cite[Thm.~6.19]{First2020}).

One might hope that every locally constant function $c: Z \to \{+1, -1\}$ is the coarse type
of some $\lambda$-involution, and thus the class of Azumaya algebras with $\lambda$-involution
up to local isomorphism and passing to matrices\uriyaf{This was missing in the original phrasing.}\benw{needs work}
is in bijection with the locally constant functions $c: Z \to \{+1, -1\}$.
However, in  Theorem~\ref{thm:unrealizableCoarseType},
we
show that this is not the case by giving examples of rings with
involution $(R,\lambda)$  for which $Z$
is disconnected, but for which any Azumaya $R$-algebra 
with $\lambda$-involution is of constant coarse type.
This answers \cite[Problem 6.24]{First2020}.
\uriyaf{end of change.}

In all of our examples, the ring $R$ is smooth over an algebraically closed
field $k$, which may be chosen arbitrarily
subject only to the requirement that the characteristic is different from $2$.
We obtain the examples by first 
observing that their desired properties follow from cohomological assumptions
on $\Spec R$, % see Propositions~\ref{pr:inclusion} and~\ref{pr:criterion},
and then constructing a ring $R$ for which these assumptions hold.

\uriyaf[inline]{Added the following remark. 
I think it may be  helpful to some potential readers.}

\begin{remark}
 The phenomena that we exhibit in this paper, i.e.,
 the existence of extraordinary involutions and  coarse types not coming from an involution,
 can occur only when the fixed locus of $\lambda^*:\Spec R\to \Spec R$
 is nontrivial, i.e., when $Z\neq \emptyset, \Spec R$.
 Indeed,  $Z$ is trivial precisely when $\lambda=\id_R$
 or $R/R^\lambda$ is quadratic \'etale, and in these situations, it is known that all $\lambda$-involutions are ordinary and all coarse types arise from $\lambda$-involutions;
 see \cite[III.\S8.1]{Knus1991} and \cite[Corollaries~6.21, 6.22]{First2020}.
 
 In addition, even when $Z\neq \emptyset, \Spec R$, it shown in \cite[Corollary~6.23]{First2020} that if $R^\lambda$ regular, then all $\lambda$-involutions are semiordinary, and it is still the case that all the coarse types arise from
        $\lambda$-involutions.  Thus, examples such as those constructed in this work can occur only when $R^\lambda$ is singular.
\end{remark}

\subsection*{Acknowledgement}

We thank an anonymous referee for many useful suggestions that have improved
and streamlined the exposition, and strengthened some statements.

\section{Involutions of central simple algebras over fields} \label{sec:fields}

Suppose $k$ is a field of characteristic different from $2$. Let $\lambda: k \to k$ be an order-$2$
automorphism, and $A$ be a degree-$n$ central simple algebra over $k$. Suppose
$\sigma : A \to A^\op$ is an involution that restricts to give $\lambda$ on $Z(A) = k \cdot
1_A$. The theory of such involutions is very well studied (see \cite{Knus1998a} for instance), and
they are all \emph{ordinary} in the terminology introduced in the introduction\uriyaf{Changed this,
  because now ordinary involutions are explained in the introduction}. We include this short section
in order to fix terminology and to emphasize the contrast between fields and general rings.

Either $\lambda=\id_k$ or $k$ is a quadratic Galois extension of the $\lambda$-fixed subfield $k_0$.

If $\lambda=\id_k$, then the involution $\sigma$ is said to be \emph{of the first kind.} Working locally in the fpqc topology on $\Spec k$ then amounts to extending scalars to an algebraic closure $\bar k$. Here $A_{\bar k}\iso \End_{\bar k}(V)$ where $V$ is a
$\bar k$-vector-space $V$. The involution $\sigma_{\bar k}$ is induced by
either a symmetric or skew-symmetric nondegenerate bilinear  form $\langle\cdot,\cdot\rangle$ on $V$ via
adjunction:
\[ \langle m(v), w \rangle = \langle v , \sigma_{\bar k}(m)(w) \rangle \qquad \forall m \in \End_{\bar k}(V), \quad
  \forall v,w \in V. \]
In the
symmetric case, $\sigma$ is said to be \emph{orthogonal} and in the skew-symmetric, it is said to be \emph{symplectic}. The
symplectic case can occur only if $n$ is even. The cases can be distinguished easily in practice, since the fixed
subspace of $\sigma$ has $k$-dimension $(n^2+n)/2$ if $\sigma$ is orthogonal and $(n^2-n)/2$ if $\sigma$ is symplectic.

If $\lambda$ is nontrivial, then $\sigma$ is said to be \emph{of the second kind}. Working locally in the fpqc topology
over $\Spec k_0$ entails passing to an algebraic closure $\bar k=\bar k_0$. Over this field, the quadratic extension
$k/k_0$ splits to give a split \'etale algebra $ \bar k \times \bar k/\bar k$, and $(A_{\bar k}, \sigma_{\bar k})$ is
isomorphic to $\Mat_{n
  \times n}(\bar k) \times \Mat_{n \times n}(\bar k)^\op$  with an involution that exchanges the two factors. Thus, there is only one   $\lambda$-involution up to fpqc-local isomorphism over $k_0$ in this case.

\section{Preliminaries}
\label{sec:conventions}
\benw[inline]{This section has been added---it was originally added after Section \ref{sec:gen}}
\uriyaf[inline]{Maybe merge Sections 2--4 into one "Conventions and Preliminaries" section?}
\benw[inline]{I think it will be good to have a short-ish ``conventions'' section. A lot of what the referee complained about was that it was hard to know what our conventions were.}
\uriyaf[inline]{I changed the organization of this section a little.}

We introduce some conventions and general notation that will be used throughout the rest of the paper.

\subsection{Conventions}

The term ``ring'' will mean ``commutative unital ring''. We write $C_2$ for the cyclic group of order $2$. The notation $I_n$ is used for the $n \times n$ identity matrix.

The letters $X$ and $Y$, and various modifications such as $\tilde X$, $Y_i$, will be used to denote schemes.
The residue field of $x\in X$ is denoted $\kappa(x)$. Given a section $f\in{\sh O}_X(U)$, where $U$ is an open subset of $X$ and $x\in U$, we write $f(x)$ for the image of $f$ in $\kappa(x)$.

By default, we use the small \'etale sites of the schemes under consideration. Therefore, the term ``sheaf'' means ``\'etale sheaf'' and ``cohomology'' means ``\'etale cohomology''. The notation  $\Hoh^n(X; \sh F)$, without further ornamentation, denotes the $n$-th \'etale cohomology group of $X$ with coefficients in the sheaf $\sh F$.

\uriyaf{Added these notations.}
Given a scheme $X$, its Picard and Brauer groups are denoted 
$\Pic(X)$ and $\Br(X)$, respectively.
Recall that the Brauer group consists of Brauer classes of Azumaya algebras over $X$,
and it is naturally a subgroup of $\Hoh^2(X,\sh O^\times_X)$ --- the cohomological Brauer group.

\uriyaf[inline]{I embedded the discussion about locality in the presence of an involution
into the subsection about good quotients, because this seems to be the only context
where we  discuss it.}

\subsection{Good quotients}
\label{sec:good-quotients}

Suppose $X$ is a scheme with an involution $\lambda$, or equivalently, an action by $C_2$. 
We must take the quotient of the scheme $X$
by the action of $C_2$, and for that purpose, we
use the notion of a \emph{good quotient} $q:X \to Y$ ultimately from \cite[Def.~1.5]{Seshadri1972} (see also the modern reference \cite[\href{https://stacks.math.columbia.edu/tag/04AB}{Tag 04AB}]{SP23}).

A good quotient of $X$ by $C_2$ is a morphism $q:X \to Y$ having the properties that $q$ is $C_2$-equivariant where $Y$ carries a trivial action, $q$ is affine and surjective, and that $\sh O_Y \to q_*(\sh O_X)$ is an isomorphism onto the $\lambda$-fixed subsheaf of $q_*(\sh O_X)$. Good quotients are categorical quotients, and therefore unique up to unique isomorphism. We will write $q: X \to X/C_2$ to denote a good quotient of $X$ by $C_2$ when it exists.

A good quotient of $X/C_2$ exists if and only if $X$ can be written as a union of $C_2$-invariant affine open subschemes, i.e., if $C_2$ acts \textit{admissibly} on $X$ in the terminology of \cite[Exp.~V, D\'ef.~1.7 \& Prop.~1.8]{Grothendieck1971}. We refer to \cite[Ex.~4.20]{First2020} for the proof. 
For example, if $X = \Spec(R)$ is affine and the $C_2$-action
arises from an involution $\lambda:R\to R$, then $\Spec(R^\lambda)$ is the good quotient $X/C_2$.

\subsection{Local Properties Involving an Involution}

Let $X$ be a scheme with an involution $\lambda$.
A $\lambda$-involution on a sheaf $\sh F$ on $X$ is a morphism of sheaves $\mu : \sh F \to \lambda_* \sh F$ with the property that $\lambda_*(\mu) \circ \mu = \id_{\sh F}$.

Suppose now that $X$ admits a good quotient $q:X\to X/C_2$ with respect to $\lambda$.
We will say that a sheaf with $\lambda$-involution $(\sh F, \mu)$ (over $X$) has some property \emph{locally} if the pair $(q_*(\sh F), q_*(\mu))$ on $X/C_2$ has this property
relative to the \'etale site of $X/C_2$.
When there is a risk of confusion regarding the base or its site, we may write ``\'etale-locally over $X/C_2$''.
Similarly, if $f: W \to X$ is a scheme over $X$ equipped with an involution $\mu: W \to W$ over $\lambda: X \to X$, then we say that $(W,\mu)$ has a property ``locally'' if it has such a property (\'etale-)locally when $\mu: W \to W$ is viewed as a diagram over $X/C_2$ by means of $q \circ f$.

\section{Involutions of Azumaya algebras} \label{sec:gen}

Let $X$ be a scheme. If $n$ is a natural number, then an Azumaya algebra of degree $n$ on $X$ is a sheaf $\sh A$ of $\sh O_X$-algebras that is \benw{changed, ``\'etale'' removed, to avoid negative inference.}locally isomorphic to the sheaf $\Mat_{n \times n}(\sh O_X)$.

One way of producing degree-$n$ Azumaya algebras on $X$ is to take a rank-$n$ locally free sheaf $\sh V$ and produce the endomorphism sheaf $\iEnd_{\sh O_X}(\sh V)$. Such algebras are called \emph{Brauer trivial}, and represent the trivial element of the Brauer group of $X$, denoted $\Br(X)$. They include the case of the trivial algebras $\Mat_{n \times n}(\sh O_X) = \iEnd(\sh O_X^n)$.

Let $\lambda : X \to X$ be an order-$2$ map.\uriyaf{Removed a sentence here as it is now superfluous.}
A \emph{$\lambda$-involution} of an Azumaya algebra $\sh A$ is an isomorphism of algebras $\tau : \sh A \to \lambda_* \sh A^{\op}$ such that $\lambda_*(\tau) \circ \tau = \id$. The notation $\sh A^\op$ denotes the algebra having the underlying module structure of $\sh A$, but where multiplication is reversed. This is the definition of involution in \cite[\S1]{Gille2009}.

Under the hypothesis that $\lambda^2 = \id$, the pushforward functor $\lambda_*$ of sheaves (of sets) and its left adjoint $\lambda^{-1}$ are naturally isomorphic.

\begin{remark} \label{rem:bundleTheory}
 We may move between locally free sheaves of finite rank and vector bundles. A
 degree-$n$ Azumaya algebra $\sh A$ over $X$ corresponds to a bundle $p : A \to X$ of algebras, locally isomorphic to
 $\Mat_{n \times n}(\sh O_X)$.

 In bundle terms, a $\lambda$-involution of algebras consists of a morphism $\tau: A \to A$ of vector bundles so that
\begin{itemize}
    \item The diagram
    \[ 
\begin{tikzcd}
    A \dar{p} \rar{\tau} & A \dar{p} \\ X \rar{\lambda} & X
\end{tikzcd}
\]
commutes;
\item $\tau^2 = \id$;
\item $\tau$ reverses the multiplication on the bundle of algebras $A$.
\end{itemize}
\end{remark}
\benw{text removed here, moved to the section on conventions}

\benw{some rewriting of the paragraph below}
An Azumaya algebra with $\lambda$-involution $(\sh A, \sigma)$ is said to be \emph{semiordinary} if there exists a Brauer-trivial Azumaya algebra $\iEnd_{\sh O_X}(\sh V)$ on $X$ and a $\lambda$-involution $\tau$ on it such that $(\sh A, \sigma)$ and $(\iEnd_{\sh O_X}(\sh V), \tau)$ are locally isomorphic (i.e., locally isomorphic on $X/C_2$ after application of $q_*$). If it is not semiordinary, then $(\sh A, \sigma)$ is \emph{extraordinary}. In an abuse of notation, we may simply say that $\sigma$ is semiordinary or extraordinary. As in the introduction, this definition is a special case of \cite[Def.~6.16]{First2020}.

\uriyaf[inline]{Move the following from the Section about ordinary involutions. I figured
it would be useful to introduce ordinary involutions here. The example is also useful for a few
other things.}

\begin{example} \label{ex:standardOrdinary}
 There is a canonical isomorphism of sheaves of rings $\sh O_X \to \lambda_* \sh O_X$, which takes $f \in \sh O_X(U)$, i.e., a section on $U$, to the composite section $f \circ \lambda \in \sh O_X(\lambda^{-1}(U))$. We will write
\[ \lambda : \sh O_X \to \lambda_* \sh O_X \]
for this isomorphism. We will also write $\lambda$ for the induced isomorphism $\Mat_{n \times n}(\sh O_X) \to \Mat_{n \times n}(\lambda_* \sh O_X)$.

    Now consider the trivial Azumaya algebra $\Mat_{n \times n}(\sh O_X)$. Let $m \in \Hoh^0(X; \GL_n(\sh O_X))$ be a global section of its group of units, and suppose further that $m$ has the property that 
    \[\lambda(m^T) = f m, \quad \text{for some $f \in \Hoh^0(X; \sh O_X^\times)$},\]
    where $m^T$ is the transpose of $m$.
    We define an involution $\tau_m$ of $\Mat_{n \times n}(\sh O_X)$ by 
    \[ \tau_m (n) = m^{-1} \lambda(n^T) m \]
    on sections. 
    One readily verifies that this is indeed a $\lambda$-involution
    of $\Mat_{n \times n}(\lambda_* \sh O_X)$.
\end{example}

A degree-$n$ Azumaya algebra with $\lambda$-involution $(\sh A, \sigma)$ is said to be \emph{ordinary} if there is
$m \in \Hoh^0(X; \GL_n(\sh O_X))$  such that  $(\sh A, \sigma)$ is
locally isomorphic  to
$(\Mat_{n \times n}(\sh O_X), \tau_m)$  defined in Example
\ref{ex:standardOrdinary}. This is the definition given in \cite[Def.~6.16]{First2020}. As with semiordinary and extraordinary algebras with involution, we may simply refer to the involution $\sigma$ as ordinary.

\uriyaf{Added this paragraph}
Ordinary involutions are semiordinary, but the converse is false in general
\cite[Example~7.7]{First2020}. 
% We will show in Remark~\ref{rm:locallyordinary} below that every $\lambda$-involution $\sigma$ is Zariski-locally ordinary.

\section{The coarse type} \label{sec:CT}
\benw{Rewrite starts here}
Suppose $X$ is a scheme endowed with a $C_2$-action, the nontrivial element acting as $\lambda:X\to X$. We write $Z$ for the equalizer of the two maps $\id_X, \lambda: X \rightrightarrows X$ and call it the\emph{ramification locus} or \emph{fixed subscheme} of $\lambda$.
Suppose moreover that a good quotient $q: X\to X/C_2$ exists. % This is equivalent to assuming that $X$ can be written as a union of $C_2$-invariant affine open subschemes, i.e., that $C_2$ acts \textit{admissibly} on $X$ in the terminology of \cite[Exp.~V, D\'ef.~1.7 \& Prop.~1.8]{Grothendieck1971}. The equivalence is explained in \cite[Ex.~4.20]{First2020}. Good quotients are categorical quotients, and therefore unique up to unique isomorphism. If $X = \Spec(R)$ is affine, then $\Spec(R^\lambda)$ is the good quotient.
\benw{Rewrite ends here.}

\textit{A priori}, $Z$ is a locally closed subscheme of $X$. A point
$x \in X$ lies in $Z$ if and only if 
$\lambda(x)=x$ and $\lambda$ induces the identity map on $\kappa(x)$.
%the morphisms $\Spec \kappa(x)\to X$ and $\Spec \kappa(x)\to X\xrightarrow{\lambda} X$.
%two evident composite morphisms $\Spec \kappa(x) \to X \rightrightarrows X$ agree. 
This means that our definition of $Z$, as a subset of $X$, agrees with the definition of \cite[\S
4.5]{First2020}
(see \cite[Prop.~4.45(d)]{First2020}). In particular $Z$ is closed as a subset of $X$. Note that in \cite[\S 4.5]{First2020}, $Z$ was endowed with the reduced induced subscheme structure, whereas here we use a functorially defined subscheme structure. 
%We further caution that our assertion that $Z$ is closed in $X$ is implicitly a consequence
%of our assumption that the good quotient $X/C_2$ exists.
In practice what
is used below is the underlying subset of $Z$, and so the scheme structure is immaterial to our arguments.

\uriyaf{Rephrase from here. The reason was that over a general ringed space, the coarse type does not
live in $\Hoh^0(Z,\mu_2)$.}
Suppose $2 \in \Hoh^0(X, \sh O_X)$ is invertible, and let $(\sh A,\sigma)$ be a degree-$n$ Azumaya algebra with $\lambda$-involution over $X$.
%and suppose that $2\in \Hoh^0(X,{\sh O}_X)$ is invertible.
In \cite[\S5.2]{First2020}, we associated to $(\sh A,\sigma)$ an invariant called
the \emph{coarse type}. This was done in the generality  
of Azumaya algebras over general ringed sites, using a descent construction. In our situation there is a simpler description of the coarse type, given in \cite[\S 5.4]{First2020}, which we now explain.

%Suppose $2\in \Hoh^0(X,{\sh O}_X)$ is invertible,
%and let $(\sh A, \sigma)$ be a degree-$n$ Azumaya algebra with $\lambda$-involution. There is an associated \emph{coarse
%  type} $c_\sigma \in \Hoh^0(Z; \mu_2)$, where $\mu_2$ is the sheaf of square-roots of $1$ in ${\sh O}_Z$. This is defined in \cite[\S
%5.2]{First2020} for Azumaya algebras over general ringed sites.

Let $\mu_2$ denote the sheaf of square-roots of $1$ in ${\sh O}_Z$.\uriyaf{End of change.}
%Here, however, we explain the coarse type using \cite[\S 5.4]{First2020}, which assumes
%$2$ is invertible on $X$. 
For a point
$z \in Z$, pulling $(\sh A, \sigma)$ back along $\Spec \kappa(z) \to X$ yields an involution of the first kind $\sigma_{\kappa(z)}$ of
the central simple $k(z)$-algebra $\sh A_{\kappa(z)}$. Define
\[ c_{\sigma}(z) = \begin{cases} +1 \quad \text{ $\sigma_{\kappa(z)}$ is an orthogonal involution on $\sh A_{\kappa(z)}$, }\\ - 1 \quad \text{ $\sigma_{\kappa(z)}$ is a symplectic involution on $\sh A_{\kappa(z)}$. }\end{cases} \]
We say that $\sigma$ is \emph{orthogonal} or \emph{symplectic at $z$} according to whether $c_\sigma(z)$ is $1$ or $-1$. The resulting function, $c_\sigma : Z \to \{+1,-1\}$ is locally constant, and therefore defines an element of $\Hoh^0(Z; \mu_2)$.
\uriyaf{Added this.}We call $c_\sigma$ the \emph{coarse type} of $\sigma$,
and $\Hoh^0(Z,\mu_2)$ the group of coarse types associated to $(X,\lambda)$.

\uriyaf{Added this paragraph. Remove if not necessary.}
In addition to the coarse type $c_\sigma:Z\to \{\pm 1\}$, it will also be convenient to
consider 
\[
Z_{\orth}( \sigma):=\{z\in Z\,:\,c_\sigma(z)=1\} 
\qquad
\text{and}
\qquad
Z_{\symp}( \sigma):=\{z\in Z\,:\,c_\sigma(z)=-1\},
\]
which are the closed (and therefore open) subschemes of $Z$ on which $\sigma$ is orthogonal and symplectic, respectively.
Clearly, $Z=Z_{\orth}( \sigma) \coprod Z_{\symp}( \sigma)$.

Perhaps surprisingly, two Azumaya algebras with $\lambda$-involution $(\sh A,\sigma)$ and $(\sh A', \sigma')$ having the same degree are locally isomorphic (over  $X/C_2$\uriyaf{Added this.}) if and only if they have the same coarse type  \cite[Thms.~5.17, 5.37]{First2020}; in plain language, they are locally isomorphic over $X/C_2$ if and only if $\sigma_{\kappa(z)}$ and $\sigma'_{\kappa(z)}$ are orthogonal (resp.~symplectic) involutions at precisely the same $\lambda$-fixed points of $X$.\benw{added the plain-language part.}

We will say that a coarse type $t \in \Hoh^0(Z, \mu_2)$ is \emph{constant} if it is constant as a function $Z \to \{\pm 1\}$. 
We will say that $t$ is \emph{realizable} if $t=c_\sigma$ for some Azumaya algebra with $\lambda$-involution $(\sh A,\sigma)$.

The coarse type of any $\lambda$-involution
of an odd-degree Azumaya algebra is identically $1$ 
because   symplectic involutions can only occur in even degrees.

\uriyaf[inline]{Added the following example and remark. Remark answers one of the referee's questions.}

\begin{example}%[compare {\cite[Example~7.1]{First2020}}]
  \label{ex:lambdatr}
 \begin{enumerate}[label=(\roman*)]
 \item  Let $n$ be a positive integer and let $\sigma:\Mat_{n\times n}(\sh O_X)\to \Mat_{n\times n}(\sh O_X)$
 be the involution $\tau_{I_n}$ in the notation of Example~\ref{ex:standardOrdinary}.
 That is, on sections, $\sigma$ is given by $\sigma(m)=\lambda (m^T)$. The coarse type $c_\sigma$ is identically $1$, since $\lambda$ restricts to the identity on $Z$, and therefore,
 for every $z\in Z$, the involution $\sigma_{k(z)}$ on $\Mat_{n\times n}(k(z))$ is simply matrix transposition.
\item Let $2r$ be an even positive integer. Let $J$ be the $2r \times 2r$ matrix
  \[ J =
    \begin{bmatrix}
      0 & I_r \\ -I_r & 0 
    \end{bmatrix}. \]
    Define $\sigma:\Mat_{2r\times 2r}(\sh O_X)\to \Mat_{2r\times 2r}(\sh O_X)$
 to be the involution $\tau_J$ in the notation of Example~\ref{ex:standardOrdinary}.
 The coarse type $c_\sigma$ is the constant function $-1$, because $\sigma$ is a symplectic involution at every $z\in Z$.
\end{enumerate}
\end{example}

\begin{remark}\label{rm:local}
 Example~\ref{ex:lambdatr} implies that the constant coarse types $\pm1\in \Hoh^0(Z,\mu_2)$ are realizable.
 
 Since the coarse type and the degree together specify an Azumaya algebra with a $\lambda$-involution up to (\'etale) local isomorphism, it also says that $\lambda$-involutions having a constant coarse type are ordinary.
 
 This implies that all $\lambda$-involutions are ordinary Zariski-locally in the following sense: if $(\sh A, \sigma)$ is an Azumaya algebra with $\lambda$-involution, then the restrictions of $(\sh A, \sigma)$ to the $C_2$-invariant open sets $X \sm Z_{\orth}(\sigma)$ and $X \sm Z_{\symp}(\sigma)$ have constant coarse type, and are therefore ordinary.\benw{lots of detail is omitted here, and placed in a comment. Anyone who actually wants to figure this out in detail can do so, I expect.}
 % The quotient map $q:X\to Y:=X/C_2$ restricts to an homeomorphism of $Z$ onto a closed subset $W$ of $Y$ \cite[Prop.~4.47(iii)]{First2020}. 
 % Let $W_{\orth}( \sigma)$ and $W_{\symp}( \sigma)$ denote
 % the images of $Z_{\orth}( \sigma)$ and $Z_{\symp}( \sigma)$ under this homeomorphism,
 % . Then $\{Y\sm W_{\orth}( \sigma), Y\sm W_{\symp}( \sigma)\}$ is
 % an open covering of $Y$
 % whose pullback to $X$ is   $\{X \sm Z_{\orth}( \sigma),X \sm Z_{\symp}( \sigma)\}$.
 % The restriction of $\sigma$ to both $X \sm Z_{\orth}( \sigma)$ and $X \sm Z_{\symp}( \sigma)$ has a constant coarse type and is therefore 
 % ordinary.\uriyaf{THis answers a question of the referee.}
\end{remark}

\section{Flat topologies}
\label{sec:flat-topologies}
Suppose $X$ is a scheme endowed with an involution $\lambda:X\to X$ and that a good quotient $q:X\to X/C_2$ exists. \benw{added the next sentence} In this section we temporarily suspend our convention that ``locally'' means ``\'etale-locally'', preferring to specify the topology in question in each instance.

We consider another notion of equivariant locality, using the big fpqc site of the quotient $X/C_2$ rather than the small \'etale site. To define this site, we use covering families $\{f_i : Y_i \to X/C_2\}_{i \in I}$ where each $f_i$ is   flat, and for each affine open $U \subseteq X/C_2$, one can find a finite subset $J \subset I$ and affine open subsets $W_i \subseteq Y_i$ for all $ i \in J$ so that $U_i = \bigcup_{i \in J} f_i(W_i)$. This is the definition of \cite[\href{https://stacks.math.columbia.edu/tag/03NV}{Tag 03NV}]{SP23} and is equivalent to that of  \cite[\S~2.3.2]{Fantechi2005}.

 We will say a sheaf with $\lambda$-involution on $X$ has a property \emph{fpqc-locally} if it has the property locally with respect to the fpqc topology on $X/C_2$. 
\uriyaf{Added a few sentences here.}In particular, one may  define the notion of an extraordinary (resp.\ ordinary, semiordinary)
involution with respect to the fpqc topology, rather than the \'etale topology
as in Section~\ref{sec:gen}; this is the definition
used in the introduction and in Section~\ref{sec:fields}.

It is noted in \cite[Rem.~4.40]{First2020} that the fppf (and therefore fpqc) topology does not have all the formal properties one would like for present purposes. We remedy this with the following proposition, which  allows us to interchange \'etale and fpqc when considering Azumaya algebras with involution, and thus use the \'etale topology exclusively in the remainder of the paper.

\begin{proposition} \label{pr:etaleLocalIsoIsFPQC}
  With notation as above, suppose\/ $2$ is invertible on $X$. Suppose given two Azumaya algebras of degree $n$ with $\lambda$-involution $(\sh A, \sigma)$, $(\sh B, \tau)$. The following are equivalent:
  \begin{enumerate}
  \item \label{pr3:i} the coarse types $c_\sigma, c_\tau: Z \to \{\pm 1\}$ agree;
  \item \label{pr3:ii} $(\sh A, \sigma)$ and $(\sh B, \tau)$ are \'etale-locally isomorphic;
  \item \label{pr3:iii} $(\sh A, \sigma)$ and $(\sh B, \tau)$ are fpqc-locally isomorphic.
  \end{enumerate}
\end{proposition}
\begin{proof}
Let $Z \subseteq X$ denote the ramification locus of $\lambda:X\to X$.
As noted in Section~\ref{sec:CT},
$(\sh A, \sigma)$ and $(\sh B, \tau)$ are \'etale-locally isomorphic if and only if they have the same coarse type, so that \eqref{pr3:i} and \eqref{pr3:ii} are equivalent.

Being fpqc-locally isomorphic is \textit{a priori} a weaker condition than being \'etale-locally isomorphic, so that \eqref{pr3:ii} implies \eqref{pr3:iii}. It suffices to show that if $(\sh A, \sigma)$ and $(\sh B, \tau)$ are fpqc-locally isomorphic, then they have the same coarse type. Since the type of the involution on $\sh A_{\kappa(x)}$ is unchanged by passing to a field extension $\kappa(x) \hookrightarrow E$, it is enough to produce, for any given point $x \in Z\subseteq X$, a map  $\eta : \Spec(E) \to  X$ with image $x$ such that $(\sh A_E, \sigma_E)$ and $(\sh B_E, \tau_E)$ are of the same type.

Suppose $\{f_i : Y_i \to X/C_2\}_{i\in I}$ is an fpqc covering with the property that $(\sh A, \sigma)$ and $(\sh B, \tau)$ become isomorphic as algebras-with-involution after pulling back along each $Y_i \times_{X/C_2} X \to X$. Since $\{Y_i \times_{X/C_2} X \to X\}_{i\in I}$ is also an fpqc covering, 
it is surjective, and so there is $i\in I$ and some point $y \in Y_i \times_{X/C_2} X$ mapping to $x$. 
Observe that $Y_i \times_{X/C_2} X$ carries a $C_2$ action by means of $\lambda$ acting on the second factor. The fixed locus of this $C_2$ action is the closed subscheme $Y_i \times_{X/C_2} Z$, since equalizers commute with pullbacks. The diagram
\[ \begin{tikzcd}
    Y_i \times_{X/C_2} Z \rar \dar & Y_i \times_{X/C_2} X \dar \\ Z \rar & X
\end{tikzcd}\]
is cartesian, which implies that the point $y$ lies in the ramification locus $Y_i \times_{X/C_2} Z$. It is meaningful to ask whether $(\sh A_{\kappa(y)}, \sigma_{\kappa(y)})$ has an orthogonal or a symplectic involution. In either case, it is of the same type as $(\sh B_{\kappa(y)}, \tau_{\kappa(y)})$, since our original two algebras-with-involution are isomorphic over $Y_i\times_{X/C_2} X$.  Consequently, the field extension $\kappa(x) \to \kappa(y)=E$ is a field extension of the kind we were seeking.
\end{proof}

\section{Coarse types and cohomology} \label{sec:ordinary}

\uriyaf[inline]{I changed the title of this section because it is no longer focused on
ordinary involutions.}

We continue to assume that $(X,\lambda)$ is a scheme with involution having a good quotient $q:X\to Y=X/C_2$. We write $Z \subseteq X$ for the fixed subscheme of the involution $\lambda:X\to X$. We assume that $2$ is invertible in $\sh O_X$.

Let $(\sh A,\sigma)$ be an Azumaya algebra with $\lambda$-involution over $X$, and let $c_\sigma\in \Hoh^0(Z,\mu_2)$
be its coarse type.
In \cite[\S 6.3]{First2020}, it was shown that certain cohomological invariants of $c_\sigma$
can determine whether $(\sh A,\sigma)$ is ordinary, semiordinary or extraordinary. In particular, these properties depend only on the coarse type.

Moreover, there is a cohomological calculation to determine whether a coarse type $c\in \Hoh^0(Z,\mu_2)$
is realizable. We recall these cohomological arguments now, and use them to give testable criteria
for the existence of extraordinary $\lambda$-involutions and non-realizable coarse types.

Our notation follows \cite[\S\S 6.2--6.3]{First2020} with the difference that we use calligraphic letters for sheaves,
e.g., our $\sh R$ is $R$ in the notation of \cite{First2020}. 
Recall our standing assumptions that a sheaf over a scheme is
a sheaf over its small \'etale site, and cohomology is \'etale cohomology.

\medskip

Let us write $\sh R$ for $q_* \sh O_X$, which is a sheaf on $Y$. The operation $\lambda$ induces an automorphism of sheaves, which we denote with the same letter: $\lambda: \sh R \to \sh R$. We also write $\sh S$ for $\sh O_Y$.

Write $\sh R^\times$ for the group of units of $\sh R$. There exists a norm map $\sh R^\times \to \sh S^\times$,  which is a map of sheaves of groups on $Y$ that is given by $x \mapsto x^\lambda x$ on sections. We denote the kernel of this map by $\sh N$. There is a map of sheaves $\sh R^\times \to \sh N$ given by $x \mapsto x^\lambda x^{-1}$ on sections. The kernel of this map is the $\lambda$-fixed subgroup of $\sh R^\times$, which is $\sh S^\times$ since $Y$ is a good quotient of $X$ by $C_2$. We define a sheaf $\sh T$ on $Y$ as the cokernel in the short exact sequence
\begin{equation}
  \label{eq:2}
  \begin{tikzcd}[column sep=4em]
    1 \rar & \sh R^\times/\sh S^\times \rar{x\mapsto x^{\lambda}x^{-1}} & \sh N \rar & \sh T \rar & 1.
  \end{tikzcd}
\end{equation}
Associated to short exact sequence \eqref{eq:2}, there is a long exact sequence in   cohomology, and in particular a connecting homomorphism
\begin{equation*} 
\delta^0 : \Hoh^0(Y; \sh T) \to \Hoh^1(Y ; \sh R^\times/ \sh S^\times).
\end{equation*}
Additionally, associated to the inclusion $\sh S^\times \to \sh R^\times$, there is a long exact sequence in cohomology and a connecting homomorphism:
\begin{equation}
\label{eq:delta1} 
\delta^1 : \Hoh^1(Y ; \sh R^\times/ \sh S^\times) \to \Hoh^2(Y; \sh S^\times).
\end{equation}
\uriyaf{Edited a little from here.}Finally,  by \cite[Thm.~5.37]{First2020}, there is a canonical isomorphism
\[ \Hoh^0(Y; \sh T) \cong \Hoh^0(Z; \mu_2) \qquad \text{(the group of coarse types).}\]
We may therefore identify   the source of $\delta^0$ with $\Hoh^0(Z;\mu_2)$ and arrive at the following diagram
\[
\Hoh^0(Z;\mu_2)\xrightarrow{\delta^0} 
\Hoh^1(Y ; \sh R^\times/ \sh S^\times) \xrightarrow{\delta^1}
\Hoh^2(Y; \sh S^\times).
\]

\benw[inline]{The paragraph below has been rewritten}
\uriyaf[inline]{Also made a few edits.}
Next, the functor $q_*$ preserves epimorphisms of groups (this is proved in \cite[Thm.~4.35(i)]{First2020}, given the definition of ``exact quotient'' in \cite[Def.~4.18]{First2020}), and as a consequence, it has a wealth of good cohomological properties. For instance, \cite[Thm.~4.23(ii)]{First2020} tells us that we can identify
\begin{equation}\label{eq:identification}
\Hoh^i(Y; \sh R^\times) = \Hoh^i(X; \sh O_X^\times) \quad \text{for all $i$.} 
\end{equation}
Composing the identification $\Hoh^i(X,\sh O_X^\times)\to \Hoh^i(Y,\sh R^\times)$
with the map $\Hoh^i(Y;\sh R^\times)\to \Hoh^i(Y;\sh O_Y^\times)$ induced
by the norm   $\sh R^\times \to \sh S^\times = \sh O_Y^\times$ 
gives the map
\[ \transf : \Hoh^i(X; \sh O_X^\times) \to \Hoh^i(Y; \sh O_Y^\times) , \]
called the \textit{cohomological $\lambda$-transfer map} in \cite[\S 6.2]{First2020},

We are now in position to state the results we need from  \cite{First2020}. Suppose $(\sh A, \sigma)$ is an Azumaya algebra with $\lambda$-involution on $X$, having coarse type $c_\sigma \in \Hoh(Z;\mu_2)\cong \Hoh^0(Y; \sh T)$ and (cohomological) Brauer class $[A] \in \Hoh^2(X; \sh O_X^\times)$, and let $c\in \Hoh^2(Z;\mu_2)$ be any coarse
type. Then:
\begin{enumerate}[label=Fact \arabic*]
\item: \label{docFactConsistency} $\delta^1 \circ \delta^0 (c_\sigma) = \transf([\sh A])$ (see \cite[Theorem 6.10]{First2020});
\item: \label{docFactOrdinary} $\sigma$ is ordinary if and only if $\delta^0(c_\sigma) = 0$;
\item: \label{docFactSemiordinary} $\sigma$ is semiordinary if and only if $\delta^1 \circ \delta^0(c_\sigma) = 0$;
\item: \label{docFactRealizable} $c$ is realizable if and only if $\delta^1\circ\delta^0(c)\in \im(\transf:\Br(X)\to \Hoh^i(Y; \sh O_Y^\times))$
.\uriyaf{We do not really need this, and one direction follows from Fact 1, but I think it is nice to include it for completeness.}
\end{enumerate}
\ref{docFactConsistency} is contained in \cite[Theorem 6.10]{First2020}
and the remaining facts are \cite[Theorem 6.19]{First2020}.
Facts 2 and 3 imply that the coarse type of a  $\lambda$-involution determines whether it is ordinary
or semiordinary. It therefore makes sense to talk about ordinary and semiordinary coarse types.

\begin{proposition} \label{pr:allOrdinary}
  If $\Pic(X)=0$, then every semiordinary involution is ordinary.
\end{proposition}
\begin{proof}
  In the notation above, we can identify
  \[ \Pic(X) = \Hoh^1(X; \sh O^\times) = \Hoh^1(Y; \sh R^\times). \]
  If this group vanishes, then $\delta^1$, which is a connecting morphism in the long exact sequence associated to
  \[ 1 \to \sh S^\times \to \sh R^\times \to \sh R^\times/\sh S^\times \to 1, \]
  is injective. Suppose $\sigma$ is a semiordinary involution. Then $\delta^1 \circ \delta^0(c_\sigma) = 0$ by  \ref{docFactSemiordinary} above. Injectivity of $\delta^1$ then implies that $\delta^0(c_\sigma) = 0$, so by \ref{docFactOrdinary}, the involution $\sigma$ is actually ordinary.
\end{proof}

\begin{prop} \label{pr:ordinaryExtends} If a  $\lambda$-involution $\sigma$ of an Azumaya algebra $\sh A$ over $X$ is ordinary, then the coarse type
  $c_\sigma \in \Hoh^0(Z; \mu_{2 })$ extends to a global section $f \in \Hoh^0(X; \sh O_X^\times)$, which satisfies
  \[ \lambda(f) = f^{-1}. \]
\end{prop}
\begin{proof}
  Since $\sigma$ is locally isomorphic to an involution as in Example \ref{ex:standardOrdinary}\uriyaf{Fixed wrong reference here.}, we may replace $(\sh A,
  \sigma)$ by $(\Mat_{n \times n}(\sh O), \tau_m)$ as in that example, without changing the coarse type. Let $f\in\Hoh(X,{\sh O}_X^\times)$ be a
  global section satisfying
  \[\lambda(m^T) = f  m. \]
  Then application of $\lambda$ and taking transpose-inverses yields
  $m^{-1} = \lambda(f^{-1}) \lambda(m^{-T})$
  which rearranges to give \[\lambda(m^T) = \lambda(f^{-1}) m. \]
  Together, the displayed equations imply 
  $\lambda(f)\cdot f=1$. This means that the pullback of $f$ along $Z\to X$, denoted
  $f_Z$,
  is a square root of $1$. By \cite[Lem.~2.17]{First_2022_octagon}, in order
  to show that $f_Z=c_\sigma$, it suffices to check that $f(z)=c_\sigma(z)$
  in $\kappa(z)$ for every $z\in Z$.
  
  Suppose $z \in Z$. Then $(\Mat_{n \times n}(\sh O_X))_{\kappa(z)} \iso \Mat_{n \times n}(\kappa(z))$ and the
  involution $\tau_{z} : =\tau_{\kappa(z)}$ is given by the formula
  \[ \tau_z (n) = m(z)^{-1} n^T m(z), \]
  where $m(z)$ is the pullback of $m$ under $\Spec \kappa(z)\to X$ That is, $m(z)$ is an $n \times n$ matrix over $\kappa(z)$ satisfying
  \[ m(z)^T = f(z) m(z). \]
  Thus $\tau_z$ is the involution associated to the nondegenerate bilinear form
  \[ B(v,w) = v^T m(z) w \]
  on $\kappa(z)^n$. 
  The form $B$ is symmetric---and so $\tau_z$ is orthogonal---if and only if $f(z) = 1$, and the form $B$ is
  skew-symmetric---and so $\tau_z$ is symplectic---if and only if $f(z) = -1$. In both cases, 
  $f(z)$ agrees with the value of the coarse type $c_\sigma$ at $z$. 
\end{proof}

Using the previous results, we can give a criterion for
a $\lambda$-involution to be extraordinary and for a coarse type not to be realizable.

\begin{proposition}\label{pr:criterion}
	Let $k$ be a field of characteristic different from $2$, let $X$ be a $k$-scheme with an involution $\lambda$ that
	is also  a $k$-morphism and let $Z$ be the ramification locus of $\lambda$. 
	Suppose that $\Hoh^0(X,\sh O^\times_X)=k^\times$
	and $\Pic(X)=0$. 
	\begin{enumerate}[label=(\roman*)]
		\item \label{i:prcriti} If the coarse type of an Azumaya algebra with $\lambda$-involution $(\sh A,\sigma)$ over
		$X$ is not constant on $Z$, then $\sigma$ is extraordinary.
		\item \label{i:pcritii} If the $2$-torsion subgroup of $\Br(X)$ is $0$\benw{changed}, then for all $\lambda$-involutions of Azumaya algebras over $X$, the coarse type is constant.
	\end{enumerate}
\end{proposition}

\begin{proof}
\ref{i:prcriti}
	It is enough to show that if the involution
	is semiordinary, then  $c_\sigma$ is constant.
	By Proposition~\ref{pr:allOrdinary} and the assumption $\Pic X=0$, the involution
	$\sigma$ is ordinary. By Proposition~\ref{pr:ordinaryExtends},
	$c_\sigma$ is the restriction of some global section $f\in \Hoh^0(X,\sh O^\times_X)$,
	which is constant by our assumptions.
	
	\ref{i:pcritii}
	Let $\sigma$ be a $\lambda$-involution on an Azumaya algebra
	$\sh A$ on $X$.
	The assumption 	$\Hoh^0(X,\sh O^\times_X) =k^\times$ implies
	in particular that $X$ is connected, so $\sh A$ must be of constant degree
	$n$, and so its Brauer class is $n$-torsion.
	
	Suppose first that $n$ is a power of $2$.
	Then our assumption on $\Br (X)$ implies that $\sh A$
	is Brauer-trivial, hence $\sigma$ is semiordinary.
	The same argument as in \ref{i:prcriti} now  shows that $c_\sigma$ is constant.
	
	In general, write $n=m2^r$ with $m$ an odd integer and $r\geq 0$.
	The involution $\sigma$ induces a $\lambda$-involution 
	$\sigma^{\otimes m}$ on the Azumaya algebra ${\sh B}={\sh A}^{\otimes m}$.
	Since ${\sh B}$ represents a $2^r$-torsion element in the Brauer group, the last
	paragraph tells us that $c_{\sigma^{\otimes m}}$ is constant.
	By checking at each point of $Z$, we see that $c_{\sigma^{\otimes m}}=c_\sigma^m$.
	Furthermore, since $m$ is odd and $\Hoh^0(Z,\mu_2)$ is $2$-torsion,
	we have $c_\sigma^m=c_\sigma$. As a result,  $c_\sigma=c_{\sigma^{\otimes m}}$ 
	is constant.
\end{proof}

We apply parts \ref{i:prcriti} and \ref{i:pcritii} of Proposition~\ref{pr:criterion} in the following two sections.
\benw{added material}Before that, we also show that supporting extraordinary $\lambda$-involutions is a Brauer-invariant property.
\begin{proposition}\label{pr:supportExtraordinaryIsBrauerInv}
  Suppose $(\sh A, \sigma)$ and $(\sh A', \sigma')$ are two Azumaya algebras with $\lambda$-involutions over $X$, and suppose $\sh A$ and $\sh A'$ are Brauer-equivalent over $X$. Then $\sigma$ is extraordinary if and only if $\sigma'$ is.
\end{proposition}
\begin{proof}
  The situation is symmetric, so it suffices to assume $\sigma$ is extraordinary and to prove that $\sigma'$ is as well. Write $t,t'$ for the coarse types of $\sigma, \sigma'$, and $[\sh A], [\sh A'] \in \Hoh^2(X; \sh O_X^\times)$ for the Brauer classes of $\sh A, \sh A'$. Using \ref{docFactConsistency} and  \ref{docFactSemiordinary} above, we see that
  \[ \transf([\sh A]) = \delta^1 \circ \delta^0 (t) \neq 0 \in \Hoh^2(Y; \sh S^\times).\]
 The equality $[\sh A ] = [\sh A']$ then gives us
  \[ \transf([\sh A']) = \transf([\sh A]) \neq 0 \]
  so that \ref{docFactConsistency} and \ref{docFactSemiordinary} imply that $(\sh A',\sigma')$ is not a semiordinary involution.
\end{proof}

\begin{remark}
  It is a shortcoming of \cite[\S 6]{First2020} that this corollary of its main results does not appear there. In truth, the authors of that paper did not believe, when the paper was written, that extraordinary involutions existed.
\end{remark}

\benw{end of added material}

\benw[inline]{I unwound additions by Uriya in the section above, since I thought the strengthening of the paper was small relative to the amount of explaining we would have to do.}

\begin{remark}\label{rm:ordinaryBreq}
  \begin{enumerate}[label=(\roman*)]
    \item If $(\sh A,\sigma)$ is an Azumaya algebra with  $\lambda$-involution,
 and $\sh A'$ is an Azumaya algebra Brauer-equivalent to $\sh A$, then
 it can happen that $\sh A'$  does not admit involutions.
 Examples and a discussion of this may be found in \cite[\S11]{First2015}
 and \cite[\S4]{Saltman1978}.

\item Proposition~\ref{pr:supportExtraordinaryIsBrauerInv} is false if we replace ``extraordinary'' with ``ordinary''. Let $(\sh B, \sigma)$ be an Azumaya algebra with an involution that is semiordinary but not ordinary, as constructed in e.g., \cite[Example~7.8]{First2020} or in Example \ref{ex:semiNotOrd} below. The algebra with involution $(\sh B, \sigma)$ is locally isomorphic (over $X/C_2$) to some Brauer-trivial algebra $(\iEnd(V), \tau)$, where $\tau$ is also semiordinary but not ordinary. But $\iEnd(V)$ itself is Brauer-equivalent to the matrix algebras $\Mat_{n \times n}(\sh O_X)$, which carry ordinary involutions.
  \end{enumerate}
\end{remark}
\section{An extraordinary involution}

We work over an algebraically closed base field $k$ of characteristic different from $2$. 
In particular, schemes are over $k$, morphisms are $k$-morphisms, products are taken over $\Spec k$ and so on.

\uriyaf[inline]{Changed everything from $2\times 2$ matrices
to general  $n\times n$ matrices for even $n$.}

Let $n$ be a positive integer, later assumed to be even.
We write $\Mat_{n \times n}$ for the affine space of $n\times n$ matrices over $k$, isomorphic to $\A^{n^2} $, and $\PGL_n$ for the $k$-algebraic group of $k$-algebra automorphisms of $\Mat_{n\times n}(k)$.

  The map $\GL_n\to \PGL_n$ sending an invertible matrix to its corresponding inner automorphism
  is surjective when both groups are regarded as sheaves on the big Zariski site of $\Spec k$.

Let $r \ge 3$ be an integer. Let $U$ be the open subvariety of $\Mat^r_{n \times n}$ whose closed points are $r$-tuples in $ \Mat_{n \times n}(k)^r$ that generate
$\Mat_{n \times n}(k)$ as a $k$-algebra. We observe that $U$ is the complement in $\Mat^r_{n \times n}$ of a closed
subvariety of codimension $r-1$; see \cite[Prop.~7.1]{First2022}, for instance.

Define two $k$-varieties
\[ \tilde Y = U \times U \times \Mat_{n \times n}, \qquad \tilde X = U \times U. \] The group $\PGL_n$ acts on
$\tilde Y$ from the right by the formula
\[ (a_1, \dots, a_r; b_1, \dots, b_r; m)\cdot [h] = (h^{-1}a_1h, \dots, h^{-1}a_rh;\, h^T b_1 h^{-T}, \dots, h^T b_r h^{-T};\,
  h^{-1} m h), \]
where here, $h$ is a lift to $\GL_n$ of $[h]$\benw{modified}.
There is a similar action on $\tilde X$, \uriyaf{added the following.}given by omitting the $m$-component. 

The group $C_2$ acts on $\tilde Y$ with the nontrivial element acting as
\[ (a_1, \dots, a; b_1, \dots, b_r; m) \mapsto (b_1, \dots, b_r; a_1, \dots , a_r; m^T ) \] and similarly for
$\tilde X$. These actions assemble to give an action of $\Gamma = \PGL_n \rtimes C_2$.  Write
$\tilde p : \tilde Y \to \tilde X$ for the projection, which is a $\Gamma$-equivariant morphism of varieties by direct
calculation.

The map $\tilde p : \tilde Y \to \tilde X$ is a trivial vector bundle. Write $\tilde{\sh A}$ for the sheaf of sections
of $\tilde p$, which is a free $\sh O_{\tilde X}$-module. Addition and multiplication in $\Mat_{n \times n}$ endow
$\tilde{\sh A}$ with the structure of an $n\times n$ matrix algebra over $\tilde X$.

\begin{proposition} \label{pr:stability}
    The varieties $\tilde X$ and $\tilde Y$ are stable, in the sense of \cite{Mumford1994}, for the $\PGL_n$-action and the trivial invertible sheaf.
\end{proposition}
\begin{proof}
    We sketch the proof for $\tilde Y$. The proof for $\tilde X$ is similar.

    For any integer $N$, write $V_N(\Mat_{n \times n})$ for the Stiefel variety of $N$-tuples of matrices $(c_1, \dots, c_N)$ that span $\Mat_{n \times n}$ as a free module.
    
    By \cite{Pappacena1997}, there is\uriyaf{Changed $5$ to $\ell$ hereon.} some natural number $l$ %($\ell=O(n^{\frac{3}{2}})$)
    such that an $r$-tuple $(a_1, \dots, a_r)$ of matrices generates $\Mat_{n \times n}(k)$ as an algebra if and only if the set of words of length $l$ in $(a_1, \dots, a_r)$ generates $\Mat_{n \times n}(k)$ as a vector space. We obtain in this way a morphism $f: \tilde Y \to V_N(\Mat_{n \times n})$ for some large $N$, given on points by sending $(a_1, \dots, a_r; b_1, \dots, b_r; m)$ to the list of words of length at most $l$ in the $a_i$ (ordered lexicographically, for instance) followed by the same words in $b_1^T,\dots,b_r^T$  and finally by $m$. 

    We embed $\PGL_n$ in the linear algebraic group $\GL(\Mat_{n \times n})$ of vector space automorphisms of $\Mat_{n \times n}$. In this way, $f : \tilde Y \to  V_N(\Mat_{n \times n})$ becomes $\PGL_n$-equivariant. The action of $\GL(\Mat_{n \times n})$ is stable on the Stiefel variety for some linearization of the trivial invertible sheaf since the Grassmannian is a geometric quotient, and \textit{a fortiori} it is stable for the $\PGL_n$-action. From here \cite[Prop. 1.18]{Mumford1994} proves what we want.
\end{proof}

\begin{proposition} \label{pr:schemThFree}
    The actions of $\PGL_n$ on the varieties $\tilde X$ and $\tilde Y$ are free and proper.
\end{proposition}
\begin{proof}
   Write $G=\PGL_n$. Both $\tilde X$ and $\tilde Y$ can be decomposed as a product of $G$-varieties $U \times Z$ for some $Z$.  In \cite[Lem. 3.2]{First2022}, it was proved that the action of $\PGL_2$ on $U$ is free and proper. The composite map
   \[  G \times U \times Z \to G \times U \times Z \times Z \to U \times U \times Z \times Z \]
   given by $(g,u,z) \mapsto (g, u, z,gz) \mapsto (u, gu, z,gz)$ is a closed embedding, being a composite of two closed embeddings, which is what had to be shown.
\end{proof}

\benw{lightly edited text below. We want to make it clear where the two propositions above are required.}
 
The stability result, Proposition \ref{pr:stability} allows us to apply \cite[Thm. 1.10]{Mumford1994} to say that there exist quasiprojective geometric quotient varieties and a morphism between them
\[ Y_1 = \tilde Y/\PGL_n \overset{p_1}{\longrightarrow}  X_1 = \tilde X/\PGL_n. \]
Since $\PGL_n$ acts on $\Mat_{n \times n}$ by algebra automorphisms, $p_1: Y_1 \to X_1$ has the structure
of an algebra bundle.

Proposition \ref{pr:schemThFree} allows us to invoke \cite[Prop.~0.9]{Mumford1994} to say that the quotient maps $\tilde Y \to Y_1$ and $\tilde X \to X_1$ are principal bundle maps and \textit{a fortiori} flat. Therefore \cite[Prop.~17.7.7]{Grothendieck1967} implies that $Y_1, X_1$ are smooth over $\Spec k$.
\benw{End of edits.}

\begin{prp}\label{pr:Azumaya-quotient}
 The sheaf of sections of $p_1:Y_1\to X_1$ is an Azumaya algebra ${\sh A}_1$ of degree $n$ over $X_1$.
\end{prp}

\begin{proof}
 It is enough to show that the pullback of $p_1:Y_1\to X_1$ along some faithfully
 flat morphism $X'\to X_1$ is isomorphic to the trivial algebra bundle $X'\times \Mat_{n\times n} \to X'$.
 We claim that the quotient morphism $\tilde X\to X_1$ fulfills this requirement. By
 \cite[Prop.~0.9]{Mumford1994}, 
 $\tilde X\to X_1$ and $\tilde Y\to Y_1$ are $\PGL_n$-torsors. This means that $\tilde X\to X_1$ is
 faithfully flat, and that the commutative square
 \begin{equation*}
      \begin{tikzcd}  \tilde Y \dar{\tilde p} \rar{} & Y_1 \dar{p_1} \\ \tilde X \rar{} & X_1
      \end{tikzcd}
    \end{equation*}
    is cartesian. Because $\tilde Y$ is $\tilde X\times \Mat_{n\times n}$, this proves the claim and the proposition.
\end{proof}

\uriyaf[inline]{Edited from here to follow referee's suggestion to shorten the proofs.
Also changed the order of the propositions a bit.}

\begin{proposition} \label{pr:C2actionOnYr}
  The $C_2$-action on $Y_1$ over $X_1$ induces the structure of a degree-$n$ Azumaya algebra with involution
  on the sheaf of sections $\sh A_1$.
\end{proposition}
\begin{proof}
    We adopt the point of view of Remark \ref{rem:bundleTheory}.

    Write $\tilde \sigma: \tilde Y \to \tilde Y$ and $\tilde \lambda : \tilde X \to \tilde X$ for the morphisms
    induced by the nontrivial element of $C_2$. It is elementary to verify that the diagram 
    \begin{equation} \label{eq:commSqLift}
        \begin{tikzcd}
            \tilde Y \dar{\tilde p} \rar{\tilde \sigma} & \tilde Y \dar{\tilde p} \\ \tilde X \rar{\tilde \lambda} & \tilde X
        \end{tikzcd}
    \end{equation}
    commutes and that all morphisms descend to $\PGL_n$-quotients. We then arrive at a commutative square:
    \begin{equation*} \label{eq:commSqDown}
        \begin{tikzcd}
            Y_1 \dar{p_1} \rar{\sigma_1} & Y_1 \dar{p_1} \\  X_1 \rar{\lambda_1} & X_1.
        \end{tikzcd}
    \end{equation*}
    Since $\tilde \sigma^2 = \id_{\tilde Y}$, it follows directly that $\sigma_1^2 =\id_{Y_1}$.
   
   It remains to verify that $\sigma_1$ is a map of vector bundles and reverses the order of multiplication. It is sufficient to check that for every $x\in X_1(k)$, the map $\sigma_1$
   restricts to an   anti-automorphism  of $k$-algebras   $p_1^{-1}(x)\to p^{-1}(\lambda_1(x))$. 
   
   As in the proof of Proposition~\ref{pr:Azumaya-quotient}, the diagram
    \begin{equation*}
      \begin{tikzcd}  \tilde Y \dar{\tilde p} \rar{f} & Y_1 \dar{p_1} \\ \tilde X \rar{h} & X_1
      \end{tikzcd}
    \end{equation*}
    is cartesian, where the horizontal maps $f,h$ are the quotient maps. If $\tilde x \in \tilde X$ is a closed point and $x \in X_1$ is its image under $h$, then $f$ induces an isomorphism of $k$-algebras between fibres $\tilde p^{-1}(\tilde x)$ and $p_1^{-1}(x)$. 
 We know  that $\tilde{\sigma}$ induces an anti-automorphism of $k$-algebras
 $\tilde{p}^{-1}(\tilde x)\to \tilde{p}^{-1}(\tilde{\lambda}(\tilde{x}))$ for every
 $\tilde{x}\in \tilde{X}(k)$, so the respective property for $\sigma_1$, $\lambda_1$, $X_1$
 holds with $x=h(\tilde x)$.  
    This is what we want.
\end{proof}

Write $\lambda_1:X_1\to X_1$
and
$\sigma_1 : \sh A_1 \to \lambda_{1*} \sh A_1$ for the involutions promised
by Proposition~\ref{pr:C2actionOnYr}.
Denote by $Z_1$ the ramification locus of $\lambda_1$.

With the goal of applying Proposition~\ref{pr:criterion}(i) to a modification
of $(\sh A_1,\sigma_1)$, which will be defined later, we now compute $\Hoh^0(X_1,\sh O_{X_1}^\times)$ 
and $\Pic(X_1)$.

\begin{proposition} \label{pr:CohoX}
  The following hold:\uriyaf{Changed the proposition to address $X_1$ instead of $X$. I think this is clearer this way.}
  \begin{enumerate}[label=(\roman*)]
  \item\label{i:Hoh0} $\Hoh^0(X_1; \sh O_{X_1}^\times) = k^\times$;
  \item \label{i:Pic} $\Hoh^1(X_1; \sh O_{X_1}^\times)=\Pic(X_1) = 0$.
  \end{enumerate}
\end{proposition}

\begin{proof}\uriyaf{This proof was rewritten following the referee's suggestions.}
  \ref{i:Hoh0} Since the quotient map $ \tilde{X}\to X_1$ is dominant, and since both $\tilde{X}$
  and $X_1$ are integral, being smooth and connected, the induced map
  $\Hoh^0(X_1,\sh O_{X_1} )\to \Hoh^0(\tilde{X},\sh O_{\tilde{X}} )$ is injective
  \cite[\href{https://stacks.math.columbia.edu/tag/0CC1}{Tag 0CC1}]{SP23}. By taking units, we deduce that $h^*:\Hoh^0(X_1,\sh O_{X_1}^\times)\to \Hoh^0(\tilde{X},\sh O_{\tilde{X}}^\times)$ is also injective, so it is enough to show that $\Hoh^0(\tilde{X},\sh O_{\tilde{X}}^\times)=k^\times$.  By
  construction, $\tilde{X}$ is the complement in $\A^{2rn^2}$ of a closed subvariety of codimension
  $r-1\geq 2$, so by the algebraic Hartog Lemma
  (\cite[\href{https://stacks.math.columbia.edu/tag/031T}{Tag
    031T}]{SP23}\uriyaf{This is the most explicit reference I could find.})
  $\Hoh^0(\tilde{X},\sh O_{\tilde{X}}^\times)=\Hoh^0(\A^{2rn^2},\sh
  O_{\A^{2rn^2}}^\times)=k^\times$.
 
  \ref{i:Pic} Sansuc \cite[Prop.~6.10]{Sansuc1981} associated with the $\PGL_n$-torsor $\tilde{X}\to X_1$ an exact sequence
 \[
 0=(\PGL_n)^{\wedge}(k)\to \Pic(X_1)\to \Pic(\tilde{X}),
 \]
 where $(\PGL_n)^{\wedge}(k)$ is the group of characters of $\PGL_n$.
 It is therefore enough  to show that $\Pic(\tilde{X})=0$.
 Since $\tilde{X}$ is smooth, we may
 identify $\Pic(\tilde{X})$ with $\CH^1(\tilde{X})$,
 and this group is $0$ because $\tilde{X}$ is open in $  \A^{2rn^2} $,
 see \cite[pp.~21, 23]{Fulton1997}.
\end{proof}

We next check that the coarse type of $\sigma_1$ is not constant when $n$ is even.

\begin{proposition} \label{pr:C2Fixed}
 Suppose $n$ is even.
  There exists a closed point $o_1 \in Z_1(k)$ where $\sigma_1$ is orthogonal and a closed point $s_1 \in Z_1(k)$ where
  it is symplectic.
\end{proposition}
\begin{proof}
 Write $m=\frac{n}{2}$.
  Let $(a_1, \dots, a_r)$ be an $r$-tuple of matrices generating $\Mat_{n \times n}(k)$---since $r \ge 2$, such an
  $r$-tuple exists. Let $w =
  \begin{bmatrix}
    0 & - I_m \\ I_m & 0 
  \end{bmatrix}$, and note $w^2 = -I_n$. Consider the points \[ o_1 = [a_1, \dots, a_r; a_1, \dots, a_r] \text{\quad and \quad }s_1 = [a_1, \dots, a_r; -wa_1w, \dots
  , -wa_rw] \]  in $X_1$. If $a \in p_1^{-1}(o_1)$ is a $k$-point, then it is possible to write $a$ as the $\PGL_2(k)$-orbit of a $(2r+1)$-tuple $(a_1, \dots,
  a_r;$ $a_1, \dots, a_r; m)$, and the involution $\sigma_1$ acts on this by sending $m \mapsto m^T$. This is an orthogonal
  involution.

Now suppose $b \in p_1^{-1}(s_1)$ is a closed point. It is possible to write $b$ as the $\PGL_2(k)$-orbit of a $(2r+1)$-tuple $(a_1, \dots,
a_r; -wa_1w, \dots, -wa_rw; m)$, and the involution $\sigma_1$ acts on this by 
\[ \begin{split} (a_1, \dots, a_r; -wa_1w, \dots, -wa_rw; m) \mapsto  (-wa_1w, \dots, -wa_rw; a_1, \dots, a_r; m^T)  \\\sim (a_1, \dots,
  a_r; -wa_1w, \dots, -wa_rw; -wm^Tw). \end{split}\]
That is, $\sigma$ acts as $m \mapsto -wm^Tw$, which is a symplectic involution.
\end{proof}

\uriyaf{Added this paragraph. We can make it into a corollary if you find it helpful.}
At this point, by applying
Proposition~\ref{pr:criterion}\ref{i:prcriti} to $(\sh A_1,\sigma_1)$ over $(X_1,\lambda_1)$,
we can assert that 
$\sigma_1$ is extraordinary when $n$ is even.
In order to obtain an extraordinary involution over an affine base scheme,
we   pull $(\sh A_1,\sigma_1)$ back along a Jouanolou construction for $X_1$.

In more detail,
since $X_1$ is quasiprojective, it carries a very ample line bundle $\sh L$, and the line bundle
$\lambda_1^*(\sh L) \tensor \sh L$ on $X_1$ remains very ample and provides an equivariant embedding of $X_1$ in some
$C_2$-equivariant projective space. That is, $X_1$ is $C_2$-quasiprojective. By \cite[Prop.~2.20]{Hoyois2017a}, there
exists a $C_2$-equivariant Jouanolou construction, viz., a $C_2$-equivariant affine vector bundle torsor $j:X \to X_1$,
where $X = \Spec R$ is a smooth affine $k$-variety carrying an involution $\lambda$.

Let $\sh A = j^* \sh A_1$, and write $\sigma$ for the involution of $\sh A$ induced by $\sigma_1$.

\begin{theorem} \label{th:AisExtraordinary}
  Suppose $n$ is even. Then the $\lambda$-involution $\sigma$ on the algebra $\sh A$ over the smooth affine $k$-variety $X$ is extraordinary.
\end{theorem}

\begin{proof}\uriyaf{Rewrote some of the proof according to the referee's suggestions.}
 By Proposition~\ref{pr:criterion}\ref{i:prcriti}, it is enough to check
 that $\Hoh^0(X,\sh O_X^\times)=k^\times$, $\Pic(X)=0$ and   the coarse type of $\sigma$
 is not constant.
 
 The first two statements were checked for $X_1$ in Proposition~\ref{pr:CohoX},
 so they will follow if we show that   $\Hoh^i(X_1,\sh O_{X_1}^\times)\to \Hoh^i(X,\sh O_X^\times)$,
 induced by $j:X\to X_1$,
 is  isomorphism for $i=0,1$.
 Since $j : X_1 \to X$ is an affine vector bundle torsor, it is Zariski locally on $X$ of the form $U \times \A^m \to U$. With this observation, it is elementary that $j^*$ induces an isomorphism
 \[ \Hoh^0(X_1; \sh O^\times_{X_1}) \to \Hoh^0(X; \sh O^\times_X). \]
To see that $j^*$ induces an isomorphism on Picard groups, use the identification of $\Pic(X)$ with $\CH^1(X)$ and \cite[Remarques p.35]{Grothendieck1958a} (see also the material of \cite[\S\S~2.1, 3.3]{Fulton1984}).

   It remains to show that the coarse type of $\sigma$ is not constant. 
   Let\uriyaf{
  The following technicality may be avoided if we would know that $j^{-1}(o)$ is fixed under $C_2$.
  I did not check in  \cite[Prop.~2.20]{Hoyois2017a} to see if that is the case, but if it is, then
  the following can be omitted.
 } $o_1,s_1\in X_1(k)$ be as in Proposition~\ref{pr:C2Fixed}. Then $C_2$ acts on the fibre $j^{-1}(o_1)\cong \A^m$.
 As observed by
 \cite[Thm.~1.2]{Serre_2009_finite_fields_for_inf_fields}, there is a $k$-point $o\in j^{-1}(o_1)$ that is fixed
 by $C_2$. Since $\lambda$ is a $k$-morphism and $\kappa(o)=k$, the point $o$ lies
 in the ramification locus $Z$ of $\lambda$. Since $j(o)=o_1$,
 we must have $(\sh A_{\kappa(o)},\sigma_{\kappa(o)})\iso
 (\sh A_{\kappa(o_1)},\sigma_{\kappa(o_1)})$, so the involution $\sigma$ is orthogonal at $o$.
 Similarly, there is a $k$-point $s\in j^{-1}(s_1)\cap Z$ at which $\sigma$ is symplectic. We conclude
 that $c_\sigma$ is not constant. 
\end{proof}

\section{Not all coarse types are realizable}
\label{sec:not-all-coarse}

Again, we work over an algebraically closed field $k$ of characteristic different from $2$.
Let $n$ be a positive integer, and let
\[ R_n = \frac{ k[x_0, \dots, x_n] }{ \left( 1 - \sum_{i=0}^n x_i^2 \right) }. \] 
We endow $R_n$ with a  $k$-involution $\lambda$
determined by 
\[ \lambda(x_i) = 
\left\{\begin{array}{ll} x_0 & i=0 
\\
-x_i & i\in\{1,\dots,n\}
\end{array}\right.\]
and write $X_n = \Spec(R_n)$.

The $\lambda$-fixed subvariety of $X_n$ is the spectrum of the quotient ring of $R_n$ by the
ideal generated by elements $r - \lambda(r)$. This ideal includes $x_1, \dots, x_n$, so that we quickly arrive at
\[ Z_n= \Spec \left(\frac{ k[x_0, \dots, x_n] }{ \left(x_1, x_2, \dots, x_n, 1 - \sum_{i=0}^n x_i^2, \right) } \right)\]
which consists of two points corresponding to $x_0 = \pm 1$ and $x_1=x_2= \dots = x_n = 0$.

The next propositions show that $X_n$ satisfies the assumptions of Proposition~\ref{pr:criterion}\ref{i:pcritii}.
\begin{proposition}\label{pr:XnTrivPic}
 If $n\geq 3$, 
 then $\Hoh^0(X_n,{\sh O}_{X_n}^\times)=k^\times$
 and $\Pic X_n=0$.
\end{proposition}

\begin{proof}
 It
 is well known that $X_n$ has the homotopy type of a motivic sphere. In the notation of \cite{Asok2017},
  $X_n$ is isomorphic to $Q_n$ and\uriyaf{Should $s$ on the first case be $2$? I did not look up the reference.}\benw{Good catch! The $2$ on the second case should be $s$. It stands for `simplicial' in distinction from the Tate circle, $\Gm$.}
  \[ Q_n \weq
    \begin{cases}
      S^{m-1}_s \wedge \Gm^m \quad \text{ if $n=2m-1$ and } \\
      S^m_s \wedge \Gm^m \quad \text{ if $n=2m$}.
    \end{cases}
  \]
  In particular, if $n \ge 3$, then using \cite[Lecture
  4]{Mazza2006} and the $\PP^1$-suspension isomorphism for motivic cohomology gives 
  \[ 0  =  \Hoh_{\mathrm{Mot.}}^2(X_n; \Z(1)) = \Pic(X_n).\]
Similarly
  \[ k^\times  =  \Hoh^1(\Spec k; \Z(1)) \isomto  \Hoh_{\mathrm{Mot.}}^1(X_n; \Z(1)) = \Hoh^0(X_n ; \sh
    O_{X_n}^\times). \qedhere \]
\end{proof}

\begin{proposition}\label{pr:BrXn-is-p-torsion}
 Let $X$ be a smooth projective quadric in $\PP^m $ ($m\geq 3$)
 and let $Y$ be a smooth hyperplane section of $X$. If $\operatorname{char} k=0$, then $\Br(X\sm Y)=0$, and if $\operatorname{char} k=p>2$, then every element of $\Br(X\sm Y)$ is annihilated by some power of $p$.
\end{proposition}

\begin{proof}
 The schemes $W$ we consider in this proof are quasiprojective and smooth over $k$,  so
 their Brauer groups $\Br(W)$ are naturally isomorphic to the cohomological Brauer group $\Hoh^2(W,\Gm)$,
 see \cite[Thms.~3.3.2, 3.5.1]{Colliot-Thelene2021}. We do not distinguish between these groups.
 
 Let $\ell$ be a prime number different from $\operatorname{char} k$,  and  let $\Br(X\sm Y)\{\ell\}$ denote the subgroup of elements whose orders are a power of $\ell$. It suffices to show that $\Br(X\sm Y)\{\ell\} = 0$.
 
 By a theorem
 of Grothendieck \cite[Cor.~6.2]{Grothendieck1968b}
 (or \cite[Thms.~3.7.1]{Colliot-Thelene2021}),
 there is an exact
 sequence
 \[
 0\to
 \Br(X)\{\ell\}
 \to
 \Br(X\sm Y)\{\ell\}
 \to
 \Hoh^1(Y,\Q_\ell /\Z_\ell).
 \]
    Note that $Y$ is irreducible because $m>2$.
 Thus, it is enough to show that $\Br(X)\{\ell\}=0$
 and $\Hoh^1(Y,\Q_\ell /\Z_\ell)=0$. 
 
 Since $k$ is algebraically closed, $X$ is  
 birationally equivalent to $\mathbb{P}^{m-1}$.\uriyaf{Fixed $n$ to $m$.} 
 As $X$ is also proper over $k$, the group
 $\Br(X)$ is the unramified cohomological Brauer group
 $\Br_{\mathrm{nr}}(k(X)/k)$, and similarly for $\mathbb{P}^{m-1}$
 \cite[Prop.~5.2.2]{Colliot-Thelene2021}.
 Since $\Br_{\mathrm{nr}}(k(\mathbb{P}^{m-1})/k)=\Br(\mathbb{P}^{m-1})=0$
 \cite[Cor. 5.1.4]{Colliot-Thelene2021},
 it follows that $\Br(X)=0$.
 
 It is well known that $\Pic(Y)$ is torsion-free, see
 \cite[Exercise~6.5(c)]{Hartshorne1977},
 and $\Hoh^0(Y, \sh O^\times_Y)=k^\times$ because $Y$ is proper.
 This and the cohomology exact sequence associated to the Kummer sequence
 $\mu_{\ell^n}\to \Gm\xrightarrow{\ell^n} \Gm$ over $Y$ together imply that $\Hoh^1(Y,\mu_{\ell^n})=0$ for all $m$,
 and hence $\Hoh^1(Y,\Q_\ell /\Z_\ell)=\colim \Hoh^1(Y,\mu_{\ell^n})=0$, where the passage to the colimit is allowed by \cite[\href{https://stacks.math.columbia.edu/tag/09YQ}{Tag 09YQ}]{SP23}. 
\end{proof}

Suppose that $n\geq 3$. Propositions~\ref{pr:XnTrivPic}
and~\ref{pr:BrXn-is-p-torsion} imply that $\Hoh^0(X_n,{\sh O}_{X_n}^\times)=k^\times$, that
$\Pic(X_n)=0$ and that \uriyaf{Slightly rephrased}the $2$-torsion of $\Br(X_n)$ is trivial.
By Proposition~\ref{pr:criterion}\ref{i:pcritii}, this means that any non-constant coarse type $c\in\Hoh^0(Z_n,\mu_2)$ is not realizable.
Since $Z_n$ consists of two points, it follows that some locally constant function $c:Z_n \to \{+1, -1\}$ 
does not arise  as the coarse type of
a $\lambda$-involution. We have proved the following theorem, which solves \cite[Problem 6.24]{First2020}. 

\begin{thm}\label{thm:unrealizableCoarseType}
 For every $n\geq 3$,   the non-constant coarse types
 in $\Hoh^0(Z_n,\mu_2)\cong \mu_2\times \mu_2$ are not realizable as the coarse type
 of a $\lambda$-involution on an Azumaya algebra over $X_n$.
\end{thm}

\begin{example}\label{ex:semiNotOrd}
  In the exceptional case of $n=2$, the ring $R_2$ supports an explicit example of a semiordinary $\lambda$-involution that is not
  ordinary. Such involutions were shown to exist in \cite[Example 7.7]{First2020}, but no explicit example was given. 

 Let $i \in k$ be a square-root of $-1$.
  We remark that $X_2=\Spec R_2$ is a Jouanolou construction over $\PP^1$
  (e.g.\ via $(x_0,x_1,x_2)\mapsto (1+x_0:x_1+ix_2):X_2\to\PP^1$),
  and in particular we have \[\Hoh^0(X_2; \sh
  O^\times_{X_2}) = \Hoh^0(\PP^1; \sh
  O^\times_{\PP^1 }) = k^\times.\]

Now consider the $2 \times 2$ idempotent matrix
  \[ E =
    \frac{1}{2}\begin{bmatrix}
      1-x_0 & x_1 + ix_2 \\ x_1 - ix_2 & 1 + x_0 
    \end{bmatrix}, \]
  which has determinant $0$ and trace $1$, and therefore has rank $1$ at all primes of $R_2$. The kernel
  of $E: R_2^2 \to R_2^2$ is a projective module $M$ of rank $1$.

  There exists an additive function $R_2^2 \to R_2^2$ given by
  \[
    \begin{bmatrix}
       r \\ s 
    \end{bmatrix} \mapsto
    \begin{bmatrix}
      \lambda(r) \\ -\lambda(s)
    \end{bmatrix}
  \]
  and this restricts to give a function $\epsilon : M \to M$ that is $\lambda$-linear in that
  \[ \epsilon(rm) = \lambda(r) \epsilon(m), \quad \forall\, r \in R_2, \, m \in M. \]
  Let $M^*=\Hom_{R_2}(M,R_2)$ be the   dual of $M$. There is a function $\epsilon^*: M^* \to M^*$ defined by
  \[ \epsilon^*(\mu) (m) = \lambda(\mu(\epsilon(m))), \quad \forall\, \mu \in M^*, \, m \in M.\]
  It is additive and $\lambda$-linear.

 Recall that the fixed locus of $\lambda$ on $X_2$
 consists of two $k$-points  corresponding to $(x_0, x_1, x_2) = (\pm 1, 0 , 0)$. At $p=(1,0,0)$,  
  the $k$-linear self-map $\epsilon(p)$ from
  $M(p):=  \kappa(p) \tensor_{R_2} M$ to itself is  the
  identity (because $E(p)=[\begin{smallmatrix} 0 & 0 \\ 0 & 1 \end{smallmatrix}]$), 
  whereas at $q=(-1,0,0)$, the map 
  $\epsilon(q) : M(q)\to M(q)$
 is multiplication by
  $-1$ (because $E(q)=[\begin{smallmatrix} 1 & 0 \\ 0 & 0 \end{smallmatrix}]$). A similar statement applies to $\epsilon^* : M^* \to M^*$: it too is the identity at $p$ and multiplication by
  $-1$ at $q$.

  Now define the $R_2$-algebra $A = \End_{R_2}( R_2 \oplus M)$. Elements of this algebra may be written as matrices
  \[
    \begin{bmatrix}
      r_1 & \mu \\ m & r_2 
    \end{bmatrix}, \qquad r_1, r_2 \in R_2, \quad m \in M, \quad \mu \in M^*,
  \]
  and the algebra structure is given by the usual matrix multiplication
  \[
      \begin{bmatrix}
      r_1 & \mu \\ m & r_2 
    \end{bmatrix}  \begin{bmatrix}
      s_1 & \nu \\ n & s_2 
    \end{bmatrix} =
    \begin{bmatrix}
      r_1s_1 + \mu(n) & r_1 \nu + s_2 \mu \\ s_1 m+r_2 n & \nu(m) + r_2s_2
    \end{bmatrix}.
  \]
  By construction, $A$ is a Brauer-trivial Azumaya algebra over $R_2$.

  We endow $A$ with a $\lambda$-involution
  \[
    \sigma \left( \begin{bmatrix}
        r_1 & \mu \\ m & r_2 
      \end{bmatrix}  \right) =
    \begin{bmatrix}
      \lambda(r_2) & \epsilon^*(\mu) \\ \epsilon(m) & \lambda(r_1)
    \end{bmatrix}.
  \]
  We calculate the coarse type of $\sigma$ by evaluating at the two fixed points $p,q$ of $X_2$. After
  base-change along $R_2 \to \kappa(p)$, the functions
  $\epsilon$ and $\epsilon^*$ become identity maps and the involution $\sigma$
  becomes
  \[ 
      \begin{bmatrix}
        r_1 & \mu \\ m & r_2
      \end{bmatrix} \mapsto
    \begin{bmatrix}
      r_2 & \mu \\ m & r_1 
    \end{bmatrix}
  \]
  which is orthogonal.
  By contrast, at $q$, the functions $\epsilon$ and $\epsilon^*$ become multiplication by $-1$, and the involution
  becomes
  \[
 \begin{bmatrix}
        r_1 & \mu \\ m & r_2
      \end{bmatrix} \mapsto
    \begin{bmatrix}
      r_2 & -\mu \\ -m & r_1 
    \end{bmatrix}
  \]
which is symplectic. Therefore the coarse type of $\sigma$ is nonconstant.

Since $\Hoh^0(X_2,{\sh O}^\times_{X_2})=k^\times$ and $c_\sigma$ is not constant, 
Proposition \ref{pr:ordinaryExtends} applied to $X_2$ implies that $\sigma$ is not ordinary. On the other hand, $A$ is Brauer-trivial, so $\sigma$ must be semiordinary.
\end{example}

\printbibliography

\end{document}

%%% Local Variables:
%%% mode: latex
%%% TeX-master: t
%%% End:

% LocalWords:  RGPIN Thm

%% file: header.tex
\usepackage{ifdraft}

\usepackage[colorlinks=true, citecolor=green!50!black, linkcolor=red!50!black, pagebackref=false, final]{hyperref}

\ifdraft{\usepackage[notref, notcite]{showkeys}}{}
\usepackage{geometry}
\usepackage{xargs}
\usepackage{amsmath, amsthm, mathrsfs, amssymb}

\usepackage{enumitem}
\usepackage{bm}
\usepackage[usenames,dvipsnames]{xcolor}
\usepackage[final]{graphicx}
\usepackage[en-AU]{datetime2}
\usepackage[T1]{fontenc}
\usepackage{tikz-cd}

\usepackage[colorinlistoftodos,prependcaption,textsize=tiny]{todonotes} % This package can be troublesome
\providecommand{\mathbx}[1]{\mathbb{#1}}

\ifdraft{

   }
   {}

\theoremstyle{plain}
 
\newtheorem{thm}{Theorem}[section]%
\newtheorem{theorem}[thm]{Theorem}

\newtheorem{prop}[thm]{Proposition}
\newtheorem{prp}[thm]{Proposition}
\newtheorem{proposition}[thm]{Proposition}

\newtheorem*{claim*}{Claim} %Claim
\newtheorem*{thm*}{Theorem}

\theoremstyle{definition}

\theoremstyle{remark}

\newtheorem{remark}[thm]{Remark}

\newtheorem{example}[thm]{Example}

\newcommand{\Q}{\mathbx{Q}}

\newcommand{\Z}{\mathbx{Z}}

\newcommand{\G}{\mathbx{G}}
\newcommand{\A}{\mathbx{A}}

\newcommand{\PP}{\mathbx{P}}

\newcommand{\Gm}{\G_m}

\newcommand{\Gl}{\operatorname{GL}}
\newcommand{\Mat}{\operatorname{Mat}}

\newcommand{\PGL}{\operatorname{PGL}}

\newcommand{\Hom}{\operatorname{Hom}}
\newcommand{\colim}{\operatornamewithlimits{colim}}

\newcommand{\tensor}{\otimes}

\newcommand{\spec}{\operatorname{Spec}}

\newcommand{\Spec}{\spec}

\newcommand{\Pic}{\operatorname{Pic}}

\newcommand{\sh}[1]{\mathcal{#1}}

\newcommand{\weq}{\simeq}

\newcommand{\op}{\text{op}}

\newcommand{\GL}{\Gl}

\newcommand{\Br}{\operatorname{Br}}

\newcommand{\End}{\operatorname{End}}

\newcommand{\Hoh}{\mathrm{H}}

\newcommand{\id}{\mathrm{id}}
\newcommand{\isom}{\cong}
\newcommand{\iso}{\isom}

\newcommand{\mto}[1]{\stackrel{#1}{\longrightarrow}}
\newcommand{\isomto}{\mto{\isom}}

\newcommand{\sm}{\setminus}

\newcommand{\im}{\operatorname{im}}

\newcommand{\CH}{{\operatorname{CH}}}